\RequirePackage{fix-cm}
\documentclass[smallextended]{svjour3}       
\smartqed  
\usepackage{graphicx}
\usepackage{amssymb,amsmath}
\usepackage{subfig}
\usepackage[english]{babel}
\usepackage{graphicx}
\usepackage{tabularx}
\usepackage{enumerate, hyperref}

\makeatletter
\def\rightharpoonupfill@{%
  \arrowfill@\relbar\relbar\rightharpoonup}
\def\leftharpoondownfill@{%
  \arrowfill@\leftharpoondown\relbar\relbar}

\newcommand{\xrightleftharpoons}[2][]{\mathrel{%
\raise.22ex\hbox{%
$\ext@arrow 3095\rightharpoonupfill@{\phantom{#1}}{#2}$}%
\setbox0=\hbox{%
$\ext@arrow 0359\leftharpoondownfill@{#1}{\phantom{#2}}$}%
\kern-\wd0 \lower.22ex\box0}%
}
\makeatother

\makeatletter
\def\rightharrowfill@{%
  \arrowfill@\relbar\relbar\rightarrow}
\def\leftharrowfill@{%
  \arrowfill@\leftarrow\relbar\relbar}

\newcommand{\xrightleftharrow}[2][]{\mathrel{%
\raise.22ex\hbox{%
$\ext@arrow 3095\rightharrowfill@{\phantom{#1}}{#2}$}%
\setbox0=\hbox{%
$\ext@arrow 0359\leftharrowfill@{#1}{\phantom{#2}}$}%
\kern-\wd0 \lower.22ex\box0}%
}
\makeatother

\begin{document}

\title{Ordering of nested square roots of 2 according to Gray code
}


\author{Pierluigi Vellucci         \and
        Alberto Maria Bersani 
}


\institute{P. Vellucci \at
              Dipartimento di Scienze di Base e Applicate per l'Ingegneria, Via Antonio Scarpa n. 16, 00161 Roma \\
              \email{pierluigi.vellucci@sbai.uniroma1.it}           
           \and
           A.M. Bersani \at
              Dipartimento di Ingegneria Meccanica e Aerospaziale, Via Eudossiana n. 18, 00184 Roma \\
              Tel.: +39-06-49766681\\
              Fax: +39-06-49766684 \\
              \email{alberto.bersani@sbai.uniroma1.it}
}

\date{Received: date / Accepted: date}

\maketitle

\begin{abstract}
In this paper we discuss some relations between zeros of Lucas-Lehmer polynomials 
and Gray code. We study nested square roots of 2 
applying a ``binary code'' that associates bits $0$ and $1$ to $\oplus$ and $\ominus$ signs in the nested form. 
This gives the possibility to obtain an ordering for the zeros of Lucas-Lehmer polynomials, which assume the form of nested square roots of 2.
\keywords{Nested radicals \and Continued radicals \and Continued roots \and Gray code \and zeros of Chebyshev polynomials}
\subclass{40A99 \and 11A99 \and 26C10}
\end{abstract}

\section{Introduction.}

Starting from the seminal paper by Ramanujan (\cite{Rama}, \cite{Ber} pp. 108-112), there is a vast literature studying the properties of the so-called continued radicals as, for example: \cite{Her,BordeB,Siz,Joh,Eft2,Ly}. In particular, Sizer \cite{Siz} determined necessary and sufficient conditions on the terms of a continued radical to guarantee convergence, while Johnson and Richmond \cite{Joh} investigated what numbers can be represented by continued radicals and if is there any uniqueness to such representation.
Nested square roots of 2 have been studied, among others, in two works of L.D. Servi \cite{Ser} and  M.A. Nyblom \cite{Nyb}. L.D. Servi tied the evaluation of nested square roots of the form
\begin{equation}
\label{eq:servi}
R(b_k,...,b_1)=\frac{b_k}{2}\sqrt{2+b_{k-1}\sqrt{2+b_{k-2}\sqrt{2+...+b_2\sqrt{2+2\sin\left(\frac{b_1 \pi}{4}\right)}}}}
\end{equation}
where $b_i\in\{-1,0,1\}$ for $i\neq 1$, to expression
\begin{equation}
\label{eq:servi2}
\left(\frac{1}{2}-\frac{b_k}{4}-\frac{b_k b_{k-1}}{8}-...-\frac{b_k b_{k-1} ... b_1}{2^{k+1}}\right)\pi
\end{equation}
to obtain some nested square roots representations of $\pi$. M.A. Nyblom, citing Servi's work, derived a closed-form expression for (\ref{eq:servi}) with a generic $x\geq2$ that replaces $\sin\left(\frac{b_1 \pi}{4}\right)$ in (\ref{eq:servi}).
Efthimiou's work \cite{Eft} proved that the radicals given by
$$a_0 \sqrt{2+a_1\sqrt{2+a_2\sqrt{2+a_3\sqrt{2+...}}}}, \ \ a_i\in\{-1,1\}$$
have limits two times the fixed points of the Chebyshev polynomials $T_{2^n}(x)$, unveiling an interesting relation between these topics. Previous formula is equivalent to (\ref{eq:prop2a}) which will be developed in the next pages.

In \cite{More,More2}, the authors report a relation between the nested square roots of depth $n$ as
\begin{equation}
\label{eq:more}
\pm  \sqrt{2\pm\sqrt{2\pm\sqrt{2\pm...\pm\sqrt{2+2 z}}}}, \ \ \ z\in\mathbb C,
\end{equation}
and the Chebyshev polynomials of degree $2^n$ in a complex variable, generalizing and unifying Servi and Nyblom's formulas, and obtaining the so-called Vi\`{e}te-like formulas (see also \cite{More3}).

In this paper we give an ordering for zeros of Lucas-Lehmer polynomials (which assume the form of nested square roots of 2 expressed by (\ref{eq:prop2a})) using the Gray code which, at the best of our knowledge, is here used to this aim for the first time. Lucas-Lehmer polynomials is a new class of polynomials introduced in \cite{Vel}, created by means of the same iterative formula, $L_n(x) = L_{n-1}^2(x) - 2$, used to build the well-known Lucas-Lehmer sequence, employed in primality tests (see for example \cite{10:10}). Although our results are similar to (\ref{eq:servi}), this approach is different because we study square roots of 2 expressed by (\ref{eq:prop2a}) applying a ``binary code'' that associates bits $0$ and $1$ to $\oplus$ and $\ominus$ signs in the nested form that expresses generic zeros of $L_{n}$.

We will now outline the content of this paper. In Section \ref{sec:prelim} we recall some important properties of Lucas-Lehmer polynomials $L_n(x)$ \cite{Vel} and of the Gray code \cite{19:19,20:20}, useful for Section \ref{sec:gray}. Here we show that the zeros of every $L_n(x)$ follow the same ordering rule of this code, where the signs $\oplus$ and $\ominus$ in the nested radicals are respectively substituted by the digits 0 e 1. In Section \ref{sec:conc} we list some further perspectives and developments of the theory. 

\section{Preliminaries.}
\label{sec:prelim}
In this section, we will introduce some useful definitions and results about Lucas-Lehmer polynomials and a particular binary code which is widely used in Informatics.

\subsection{Lucas-Lehmer polynomials.}
We recall some basic facts about Lucas-Lehmer polynomials, which follow the iterative formula
\begin{equation}
L_0(x) = x \quad ; \quad L_n(x) = L_{n-1}^2 - 2 \ \ \ \ \ \forall n \geq 1
\end{equation}
and their zeros, whose proofs can be found in \cite{Vel}.

Assuming $L_{0}=x$ as the initial value, let us construct the first terms of the sequence.
The function $L_{1}(x)=x^{2}-2$ represents a parabola with two zeros $z_{1,2}=\pm \sqrt{2}$ and one minimum point in $(0,-2)$; $L_{2}(x)=(x^{2}-2)^{2}-2= 2 \left(1- 2 x^2 + \frac{x^4}{2} \right)$, shown in Fig. \ref{fig:L_2} contains four zeros: $z_{1\div 4}=\pm \sqrt{2 \pm \sqrt{2}}$. From the derivative of $L_{2}(x)$, $L_{2}'(x)=4x\cdot (x^{2}-2)=4x\cdot L_{1}(x)$ it is possible to determine the critical points of the function: $x_1=0$ (minimum), $x_{2,3}=\pm \sqrt{2}$ (maximum).\
Since $L_{2}(x)= 2 \left(1- 2 x^2 + \frac{x^4}{2} \right) = 2 \cos(2x)+ \textit{o}(x^3)$, for $x\rightarrow 0$ we have $L_{2}(x)\sim 2\cos(2x)$.
\begin{figure}[tb]
\centering
\includegraphics[scale=0.30]{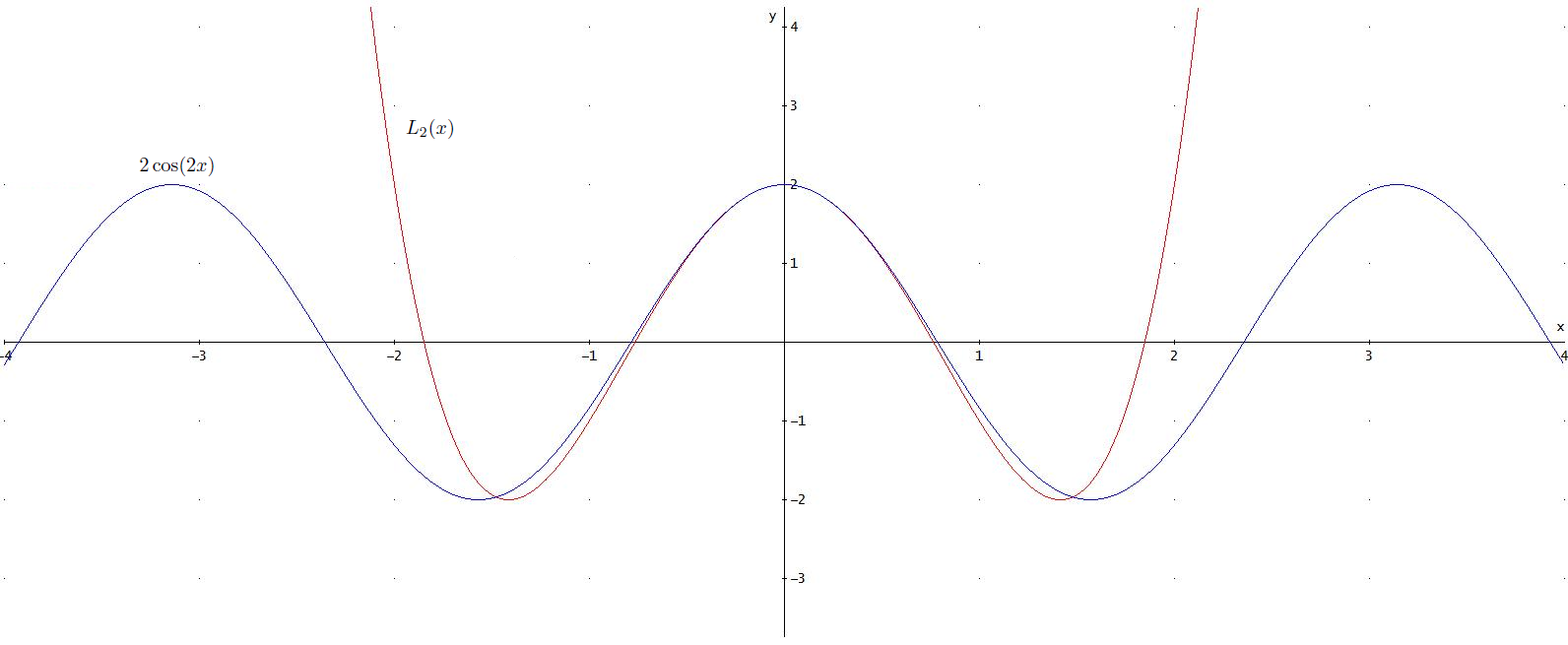}
\caption{comparison between $L_{2}(x)$ and $2\cos(2x)$.}
\label{fig:L_2}
\end{figure}

\begin{figure}[tb]
\centering
\includegraphics[scale=0.30]{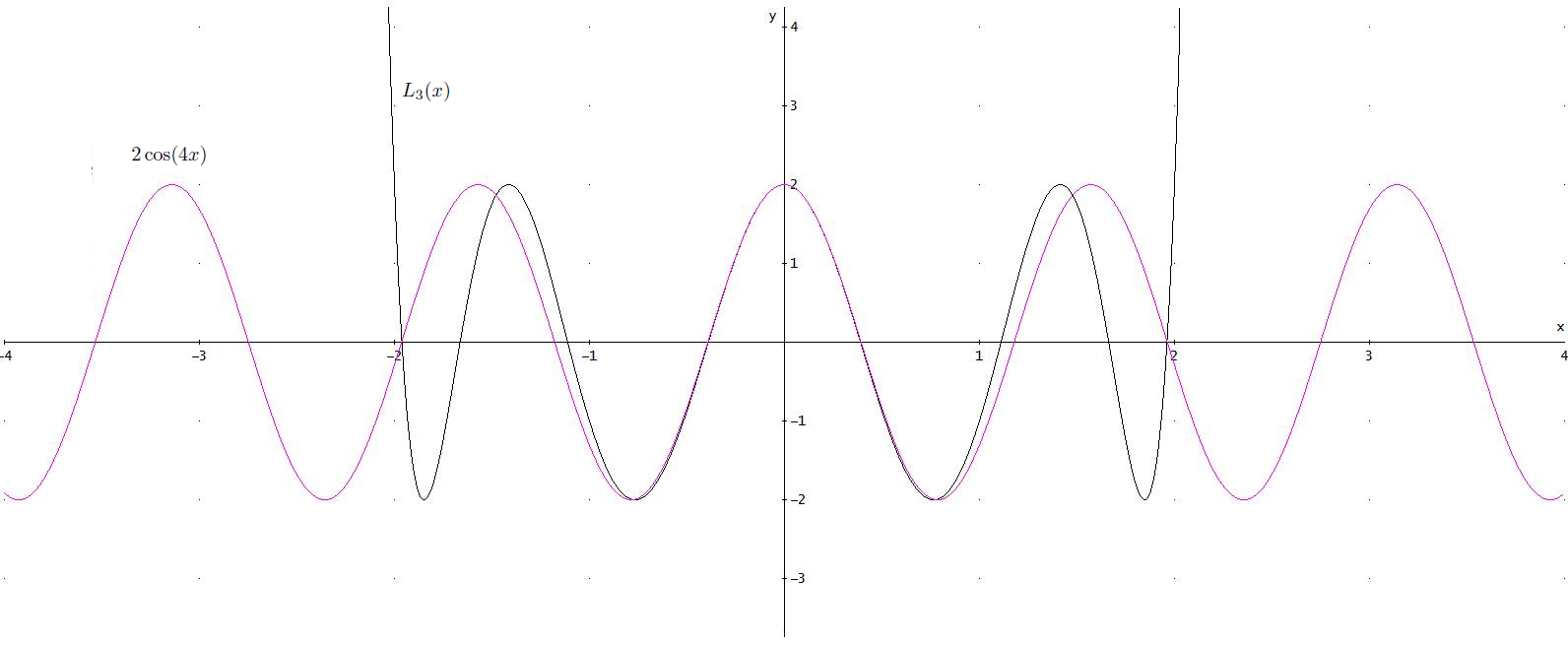}
\caption{comparison between $L_{3}(x)$ and $2\cos(4x)$.}
\label{fig:L_3}
\end{figure}
The zeros of the function $L_{3}(x)=((x^{2}-2)^{2}-2)^{2}-2 = 2 \left(1 - 8x^2 + \textit{o}(x^3)\right)$, whose graph is shown in Fig. \ref{fig:L_3}, are eight: $z_{1\div 8}=\pm \sqrt{2 \pm \sqrt{2 \pm \sqrt{2}}}$. The critical points are: $x_1 = 0$, $x_{2,3}= \pm \sqrt{2}$, $x_{4,5,6,7}=\pm \sqrt{2 \pm \sqrt{2}}$. Besides $L_{3}(x)\sim 2\cos(4x)$ for $x\rightarrow 0$.
\begin{figure}[tb]
\centering
\includegraphics[scale=0.30]{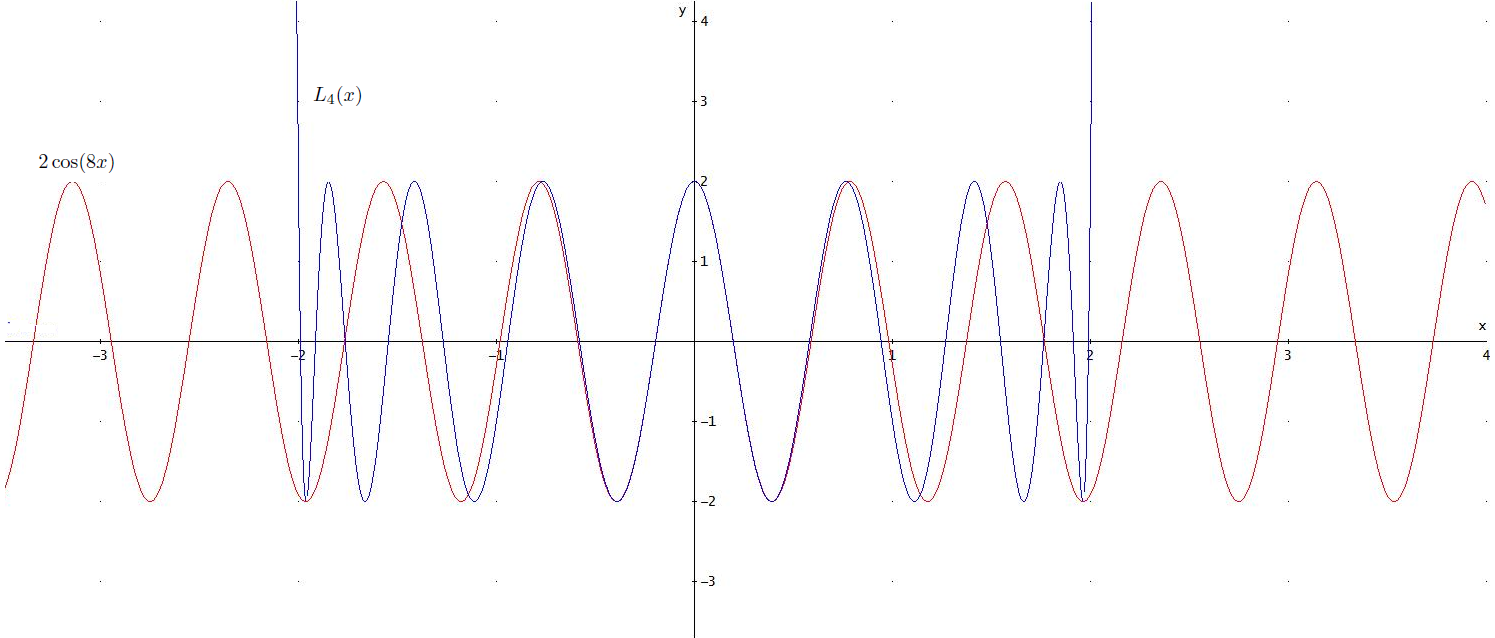}
\caption{comparison between $L_{4}(x)$ and $2\cos(8x)$.}
\label{fig:L_4}
\end{figure}
The zeros of the function (shown in Fig. \ref{fig:L_4}) $L_{4}(x)=(((x^{2}-2)^{2}-2)^{2}-2)^{2}-2$ are sixteen: $z_{1\div 16}=\pm \sqrt{2 \pm \sqrt{2 \pm \sqrt{2 \pm \sqrt{2}}}}$. The critical points follow the same general rule which is possible to guess observing the previous iterations; moreover it results again $L_{4}(x)\sim 2\cos(8x)$ for $x\rightarrow 0$.
It must be noted that: $L_1(\pm \sqrt{2}) = 0$, $L_2(\pm \sqrt{2}) = - 2$, $L_n(\pm \sqrt{2}) = 2 \quad \forall n \ge 3$; $L_0(0) = 0$, $L_1(0) = - 2$, $L_n(0) = 2 \quad \forall n \ge 2$; $L_0(-2) = -2$, $L_n(-2) = 2 \quad \forall n \ge 1$; $L_n(2) = 2 \quad \forall n \ge 0$.

\begin{proposition}
\label{p:1}
At each iteration the zeros of the map $L_n (n \geq 1)$ have the form
\begin{equation}
\label{eq:prop2a}
\pm  \sqrt{2\pm\sqrt{2\pm\sqrt{2\pm\sqrt{2\pm...\pm\sqrt{2}}}}}
\end{equation}
\end{proposition}

Let $M_{n}$ be the set of the critical points and be $Z_{n}$ the set of the zeros of $L_{n}(x)$; we have the following results.
\begin{proposition}
\label{propo:seconda}
For each $n\geq 2$ we have
\begin{equation}
\label{eq:M_L}
M_{n}=M_{n-1}\cup Z_{n-1}=M_{1}\cup \bigcup_{i=1}^{n-1} Z_{i}
\end{equation}
with $card(Z_{n})=2^{n}$.
\end{proposition}

\begin{corollary}
\label{c:1}
All zeros and critical points of $L_n$ belong to the interval $(-2, 2)$ \footnote{Because of the symmetry of Lucas-Lehmer polynomials, we will study only positive zeros.}.
\end{corollary}

\begin{proposition}
For each $n\geq1$ we have
\begin{equation}
\label{eq:propcheby1}
L_{n}(x)=2\ T_{2^{n-1}}\left(\frac{x^{2}}{2}-1\right)
\end{equation}
where $T_n$ are the Chebyshev polynomials of the first kind.
\end{proposition}

\begin{proposition}
\label{cor5bis}
The polynomials $L_n(x)$ are orthogonal with respect to the weight function
$$\frac{1}{4\sqrt{4-x^2}}$$
defined on $x\in(-2,2)$.
\end{proposition}

\subsection{Gray code.}

Given a binary code, its \textbf{order} is the number of bits with which the code is built, while its \textbf{length} is the number of strings that compose it. The celebrated Gray code \cite{19:19,20:20} is a binary code of order $n$ and length $2^{n}$.
\begin{figure}[tb]
\centering
\includegraphics[scale=0.80]{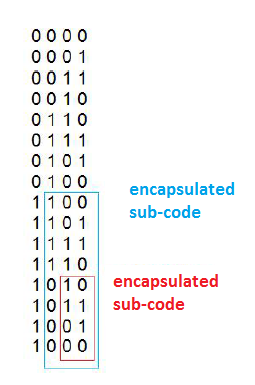}
\caption{Sub-codes for $m=2$, $m=3$.}
\label{fig:gray}
\end{figure}

We briefly recall below how a Gray Code is generated; if the code for $n-1$ bits is formed by binary strings
\begin{align}
\label{eq:n-1_gray}
&g_{n-1,1}\notag \\
&...\notag \\
&g_{n-1,2^{n-1}-1}\notag \\
&g_{n-1,2^{n-1}}
\end{align}
the code for $n$ bits is built from the previous one in the following way:
\begin{align}
\label{eq:n_gray}
&0g_{n-1,1}\notag \\
&...\notag \\
&0g_{n-1,2^{n-1}-1}\notag \\
&0g_{n-1,2^{n-1}}\notag \\
&1g_{n-1,2^{n-1}}\notag \\
&1g_{n-1,2^{n-1}-1}\notag \\
&...\notag \\
&1g_{n-1,1}\notag \\
\end{align}

Just as an example, we have

\noindent
for $n=1$: $g_{1,1}=0 \ ; \ g_{1,2}=1$;

\noindent
for $n=2$: $g_{2,1}=00 \ ; \ g_{2,2}=01 \ ; \ g_{2,3}=11 \ ; \ g_{2,4}=10$

\noindent
for $n=3$: $g_{3,1}=000 \ ; \ g_{3,2}=001 \ ; \ g_{3,3}=011 \ ; \ g_{3,4}=010 \ ; \ g_{3,5}=110 \ ; \ g_{3,6}=111 \ ; \ g_{3,7}=101 \ ; \ g_{3,8}=100$

\noindent
and so on.

\begin{definition}
\label{def:prima}
Let us consider a Gray code of order $n$ and length $2^{n}$. A \emph{sub-code} is a Gray code of order $m<n$ and length $2^{m}$.
\end{definition}
\begin{definition}
\label{def:seconda}
Let us consider a Gray code of order $n$ and length $2^{n}$. An \emph{encapsulated sub-code} is a sub-code built starting from the last string of the Gray code of order $n$ that contains it.
\end{definition}
Figure (\ref{fig:gray}) contains some examples of encapsulated sub-codes inside a Gray code (with order $4$ and length $16$).

\section{Gray code and nested square roots.}
\label{sec:gray}

It is known (Propositions \ref{p:1}, \ref{propo:seconda} and Corollary \ref{c:1}) that $L_{n}$ has $2^{n}$ zeros, symmetric with respect to the origin. Let us consider the signs $\oplus, \ominus$ in the nested form that expresses generic zeros of $L_{n}$, 
as follows:
\begin{equation}
\label{eq:annidata}
\sqrt{2\pm\underbrace{\sqrt{2\pm\sqrt{2\pm\sqrt{2\pm...\pm\sqrt{2\pm\sqrt{2}}}}}}}
\end{equation}
Obviously the underbrace encloses $n-1$ signs $\oplus$ or $\ominus$, each one placed before each nested radical. Starting from the first nested radical we apply a code (i.e., a system of rules) that associates bits $0$ and $1$ to $\oplus$ and $\ominus$ signs, respectively.

Let us define with $\displaystyle \{\omega(g_{n-1}, j) \}_{j=1, ..., 2^{n-1}}$ the set of all the $2^{n-1}$ nested radicals of the form
\begin{equation}
2\pm\underbrace{\sqrt{2\pm\sqrt{2\pm\sqrt{2\pm...\pm\sqrt{2\pm\sqrt{2}}}}}}_{n-1 \ \ signs}=\omega(g_{n-1, 1 \div 2^{n-1}}) \ ,
\end{equation}
where each element of the set differs from the others for the sequence of $\oplus$ and $\ominus$ signs.
Then:
\begin{equation}
\sqrt{2\pm\underbrace{\sqrt{2\pm\sqrt{2\pm\sqrt{2\pm...\pm\sqrt{2\pm\sqrt{2}}}}}}_{n-1 \ \ signs}}=\sqrt{\omega(g_{n-1, 1 \div 2^{n-1}})}
\end{equation}
where the notation $n-1, 1 \div 2^{n-1}$ means that it is possible obtain $2^{n-1}$ strings formed by $n-1$ bit.
\begin{theorem}
\label{theo:graycode}
The strings with which we code the $2^{n-1}$ positive zeros of $L_{n}$ (sorted in decreasing order) follow the sorting of Gray code. That is, if
\begin{align}
&g_{n-1,1}\notag \\
&...\notag \\
&g_{n-1,2^{n-1}-1}\notag \\
&g_{n-1,2^{n-1}}
\end{align}
is the Gray Code, then
\begin{equation}
\label{eq:theorem_graycode}
\sqrt{\omega(g_{n-1,1})} > ... > \sqrt{\omega(g_{n-1,2^{n-1}-1})} > \sqrt{\omega(g_{n-1,2^{n-1}})}
\end{equation}
\end{theorem}
\begin{proof}
We first prove ($\ref{eq:theorem_graycode}$) for $n=2$; here the Gray Code is reduced to bits $0,1$ (with this order). Indeed we have $\sqrt{\omega(0)}>\sqrt{\omega(1)}$ because:
\begin{equation}
\sqrt{2+\sqrt{2}}>\sqrt{2-\sqrt{2}} \Leftrightarrow 2+\sqrt{2} > 2-\sqrt{2} \Leftrightarrow 2\sqrt{2}>0
\end{equation}
Let us now suppose that ($\ref{eq:theorem_graycode}$) is true for the Gray Code of order $n-1$.
We know that
\begin{equation}
z_{(n)}=\pm\sqrt{2\pm z_{(n-1)}}
\end{equation}
where $z_{(n)}$ and $z_{(n-1)}$ are the generic zeros of $L_{n}$ and $L_{n-1}$. For the symmetry of the zeros we can  consider only positive zeros. Therefore
\begin{equation}
\label{eq:relazionezeri}
z_{(n)}=\sqrt{2\pm z_{(n-1)}}
\end{equation}
But the generic zero of $L_{n-1}$, according to the hypothesis, is precisely one among $\sqrt{\omega(g_{n-1,1})}$,... ,$\sqrt{\omega(g_{n-1,2^{n-1}-1})},\sqrt{\omega(g_{n-1,2^{n-1}})}$; then the generic zero can be indicated with $\sqrt{\omega(g_{n-1,1\div 2^{n-1}})}$, in a more compact form. Then, from ($\ref{eq:relazionezeri}$) we can separate the cases $\oplus$ and $\ominus$, obtaining either
\begin{equation}
z_{(n)}=\underbrace{\sqrt{2 + \sqrt{\omega(g_{n-1,1\div 2^{n-1}})}}}_{\sqrt{\omega(0g_{n-1,1\div 2^{n-1}})}}
\end{equation}
because $\oplus$ corresponds to 0, or
\begin{equation}
z_{(n)}=\underbrace{\sqrt{2 - \sqrt{\omega(g_{n-1,1\div 2^{n-1}})}}}_{\sqrt{\omega(1g_{n-1,1\div 2^{n-1}})}}
\end{equation}
because $\ominus$ corresponds to 1. Thesis follows if we show the following:
\begin{align}
\label{eq:induz_bpiu1_generico}
&\sqrt{\omega(0,g_{n-1,1})}>\sqrt{\omega(0,g_{n-1,2})}>...>\sqrt{\omega(0,g_{n-1,2^{n-1}})}> \notag \\
&>\sqrt{\omega(1,g_{n-1,2^{n-1}})}>...>\sqrt{\omega(1,g_{n-1,1})}
\end{align}
or equivalently
\begin{align}
&\omega(0,g_{n-1,1})>\omega(0,g_{n-1,2})>...>\omega(0,g_{n-1,2^{n-1}})> \notag \\
&>\omega(1,g_{n-1,2^{n-1}})>...>\omega(1,g_{n-1,1})
\end{align}
We start proving inequality
\begin{equation}
\label{eq:prima_dis}
\omega(0,g_{n-1,i})>\omega(0,g_{n-1,i+1}) \ \ \ \ \ \ \ \ \ \forall i=1,2,...,2^{n-1}-1
\end{equation}
or
\begin{align}
\label{eq:dim_primo_step-prima_dis}
&2+\sqrt{\omega(g_{n-1,i})}>2+\sqrt{\omega(g_{n-1,i+1})} \Leftrightarrow \notag \\
&\Leftrightarrow  \omega(g_{n-1,i})>\omega(g_{n-1,i+1})
\end{align}
true by virtue of hypothesis ($\ref{eq:theorem_graycode}$). We prove now the inequality
\begin{equation}
\label{eq:secondaseconda_dis}
\omega(1,g_{n-1,i+1})>\omega(1,g_{n-1,i}) \ \ \ \ \ \ \ \ \ \forall i=1,2,...,2^{n-1}-1
\end{equation}
It can be rewritten
\begin{equation}
\label{eq:dim_primo_step-seconda_dis}
2-\sqrt{\omega(g_{n-1,i+1})}>2-\sqrt{\omega(g_{n-1,i})} \Leftrightarrow  \omega(g_{n-1,i})>\omega(g_{n-1,i+1})
\end{equation}
true for assumption.
Finally, the relation
\begin{equation}
\label{eq:terza_dis}
\omega(0,g_{n-1,2^{n-1}})>\omega(1,g_{n-1,2^{n-1}})
\end{equation}
follows naturally from
\begin{align}
\label{eq:dim_primo_step-terza_dis}
&2+\sqrt{\omega(g_{n-1,2^{n-1}})}>2-\sqrt{\omega(g_{n-1,2^{n-1}})} \Rightarrow \notag \\
&\Rightarrow 2\sqrt{\omega(g_{n-1,2^{n-1}})}>0
\end{align}
always true.
\end{proof}

\begin{figure}[tb]
\centering
\includegraphics[scale=0.50]{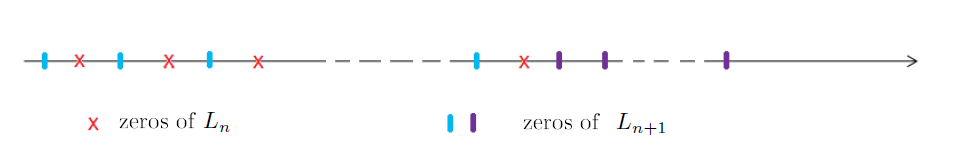}
\caption{Disposition of the zeros of $L_{n}$ and $L_{n+1}$ on the real axis.}
\label{fig:zeriL}
\end{figure}
In Table \ref{tab:1} we give an example of ordering of the zeros of $L_n(x)$ for $n=4$.
\begin{center}
\begin{tabular}{p{0.7cm}p{1.5cm}p{4cm}p{1.5cm}}
$g_{3,j}$&Binary string&\ \ \ \ \ \ \ \ \ \ \ \ \ \ Radicals&Approx.\\ \hline
$g_{3,1}$&000&
$\sqrt{2+\sqrt{2+\sqrt{2+\sqrt{2}}}}$&$1.99\dots$\\ \hline
$g_{3,2}$&001 &
$\sqrt{2+\sqrt{2+\sqrt{2-\sqrt{2}}}}$&$1.91\dots$\\ \hline
$g_{3,3}$&011 &
$\sqrt{2+\sqrt{2-\sqrt{2-\sqrt{2}}}}$&$1.76\dots$\\ \hline
$g_{3,4}$&010 &
$\sqrt{2+\sqrt{2-\sqrt{2+\sqrt{2}}}}$&$1.54\dots$\\ \hline
$g_{3,5}$&110 &
$\sqrt{2-\sqrt{2-\sqrt{2+\sqrt{2}}}}$&$1.26\dots$\\ \hline
$g_{3,6}$&111 &
$\sqrt{2-\sqrt{2-\sqrt{2-\sqrt{2}}}}$&$0.94\dots$\\ \hline
$g_{3,7}$&101 &
$\sqrt{2-\sqrt{2+\sqrt{2-\sqrt{2}}}}$&$0.58\dots$\\ \hline
$g_{3,8}$&100 &
$\sqrt{2-\sqrt{2+\sqrt{2+\sqrt{2}}}}$&$0.19\dots$\\ \hline
\end{tabular}
\captionof{table}{In the Table we consider the $2^3$ positive zeros of $L_4(x)$ and their ordering due to Gray code. It is according to Theorem \ref{theo:graycode}.}
\label{tab:1}
\end{center}

Since, as shown in Proposition \ref{cor5bis}, $\{L_n\}$ is a set of orthogonal polynomials, it follows that between two zeros of $L_n(x)$ there exists one and only one zero of $L_{n+1}(x)$ (see \cite{13:13}).
\begin{theorem}
\label{theo:zeridisposizione}
Let us consider the $2^{n-1}$ zeros of $L_{n}(x)$
\begin{equation}
\sqrt{\omega(g_{n-1,2^{n-1}})}<\sqrt{\omega(g_{n-1,2^{n-1}-1})}< ... <\sqrt{\omega(g_{n-1,1})} \ .
\end{equation}
Then the zeros of $L_{n+1}(x)$ are arranged on the real axis in this way:
\begin{itemize}
  \item {\bf i)} The first zero of $L_{n+1}(x)$ (i.e. $\sqrt{\omega(1,g_{n-1,1})}$) is on the left of the first zero of $L_{n}(x)$: $\sqrt{\omega(1,g_{n-1,1})}<\sqrt{\omega(g_{n-1,2^{n-1}})}$.
  \item {\bf ii)} The $2^{n-1}-1$ zeros of $L_{n+1}(x)$, which can be represented in the form $$\sqrt{\omega(1,g_{n-1,2\div 2^{n-1}})},$$ are arranged one by one in the $2^{n-1}-1$ intervals which have consecutive zeros of $L_{n}(x)$; i.e.: $(\sqrt{\omega(g_{n-1,k})},\ \sqrt{\omega(g_{n-1,k-1})})$.
  \item {\bf iii)} The remaining zeros, expressed as $\sqrt{\omega(0,g_{n-1,1\div 2^{n-1}})}$, are on the right of the last zero of $L_{n}(x)$:$\sqrt{\omega(g_{n-1,1})}$.
\end{itemize}
The above is schematically shown in Figure (\ref{fig:zeriL}).
\end{theorem}
\begin{proof}
For the proof we first need to dispose on the real axis the $2^{n-1}$ zeros of $L_{n}$; in the interval $J=(a,b)$, where $a$ and $b$ are the first and the last zeros:
\begin{equation}
J=\Bigl(\sqrt{\omega(g_{n-1,2^{n-1}})},\ \sqrt{\omega(g_{n-1,1})}\Bigr)
\end{equation}
we can identify $2^{n-1}-1$ subintervals
\begin{equation}
J_{k,k-1}=\Bigl(\sqrt{\omega(g_{n-1,k})},\ \sqrt{\omega(g_{n-1,k-1})}\Bigr) \ \ \ k=2, 4, 8, ... , 2^{n-1}
\end{equation}
whose endpoints are consecutive zeros of $L_{n}$.
Since $L_{n}$ is a \emph{Tchebycheff} polynomial, and therefore it is an orthogonal polynomial, it follows that between two consecutive zeros of $L_{n}(x)$ there exists one and only one zero of $L_{n+1}(x)$. Therefore
in each interval $J_{k, k-1}$ we find only one zero of $L_{n+1}(x)$. Let us understand how they are distributed. Let us start with the first $2^{n-1}$ zeros of $L_{n+1}(x)$, i.e.:
\begin{equation}
\sqrt{\omega(1,g_{n-1,1})}<\ ...\ <\sqrt{\omega(1,g_{n-1,2^{n-1}})}
\end{equation}
The statements {\bf i)} and {\bf ii)} are true if we show that
\begin{equation}
\label{eq:primadis}
\sqrt{\omega(1,g_{n-1,1})}<\sqrt{\omega(g_{n-1,2^{n-1}})}
\end{equation}
and
\begin{equation}
\label{eq:secondaadis}
\sqrt{\omega(1,g_{n-1,2})}>\sqrt{\omega(g_{n-1,2^{n-1}})}
\end{equation}
In fact, let us recall that
\begin{align}
&\sqrt{\omega(1,g_{n-1,1})} \ \ {\rm is \ related \ to \ the \ sequence} \ \ 1\underbrace{0\ ...\ 0}_{n-1} \notag \\
&\sqrt{\omega(g_{n-1,2^{n-1}})} \ \ {\rm is \ related \ to \ the \ sequence} \ \ 1\underbrace{0\ ...\ 0}_{n-2} \notag \\
&\sqrt{\omega(1,g_{n-1,2})}\ \ {\rm is \ related \ to \ the \ sequence} \ \ 1\underbrace{0\ ...\ 0}_{n-2}1
\end{align}
From the first two relations, we can show (\ref{eq:primadis}):
\begin{align}
&\sqrt{\omega(1,g_{n-1,1})}<\sqrt{\omega(g_{n-1,2^{n-1}})} \ \ \Leftrightarrow \sqrt{2-\sqrt{2\underbrace{+\sqrt{2+\sqrt{2+...+\sqrt{2}}}}_{n-1}}}< \notag \\
&<\sqrt{2-\sqrt{2\underbrace{+\sqrt{2+\sqrt{2+...+\sqrt{2}}}}_{n-2}}} \Leftrightarrow \sqrt{2\underbrace{+\sqrt{2+\sqrt{2+...+\sqrt{2}}}}_{n-1}}> \notag \\
&>\sqrt{2\underbrace{+\sqrt{2+\sqrt{2+...+\sqrt{2}}}}_{n-2}}
\end{align}
noting that
\begin{equation}
\sqrt{2\underbrace{+\sqrt{2+\sqrt{2+...+\sqrt{2}}}}_{n-1}}=\sqrt{2+\sqrt{2\underbrace{+\sqrt{2+...+\sqrt{2}}}_{n-2}}}
\end{equation}
is greater than
\begin{equation}
\sqrt{2\underbrace{+\sqrt{2+\sqrt{2+...+\sqrt{2}}}}_{n-2}} \ .
\end{equation}
Let us prove (\ref{eq:secondaadis}) by the same reasoning used previously.
\begin{align}
&\sqrt{\omega(1,g_{n-1,2})}>\sqrt{\omega(g_{n-1,2^{n-1}})} \ \ \Leftrightarrow \sqrt{2-\sqrt{2\underbrace{+\sqrt{2+\sqrt{2+...-\sqrt{2}}}}_{n-1}}}> \notag \\
&>\sqrt{2-\sqrt{2\underbrace{+\sqrt{2+\sqrt{2+...+\sqrt{2}}}}_{n-2}}} \ \ \Leftrightarrow \sqrt{2\underbrace{+\sqrt{2+\sqrt{2+...-\sqrt{2}}}}_{n-1}}< \notag \\
&<\sqrt{2\underbrace{+\sqrt{2+\sqrt{2+...+\sqrt{2}}}}_{n-2}}
\end{align}
By squaring iteratively the $n-2$ radicals and simplifying, we obtain $-\sqrt{2}<0$, which is true.

{\bf iii)} follows immediately noting that
\begin{equation}
\omega(g_{n-1,1}) < 2  \quad ; \quad \omega(0, g_{n-1, 2^{n-1}}) = 2 + \sqrt{\omega(g_{n-1, 2^{n-1}})} > 2 \ .
\end{equation}
The three points of the thesis are proven.
\end{proof}
\begin{remark}
In the previous work \cite{Vel}, the considerations made on the map $L_n$ were extended to an entire class of maps, obtained through the iterated formula $M^a_{n}=2a \left(M^a_{n-1}\right)^{2}-\frac{1}{a} \ , \ a >0$, with $M^a_0(x) = x$. At each iteration the zeros of the map $M^a_n (n \geq 1)$ have the form
$$\pm \frac{1}{2a}\cdot \sqrt{2\pm\sqrt{2\pm\sqrt{2\pm\sqrt{2\pm...\pm\sqrt{2}}}}}$$
and it is clear that the results obtained in this paper are also valid for polynomials obtained through the iterated formula on $M^a_{n}$.
\end{remark}

\section{Conclusions and perspectives.}
\label{sec:conc}

In this paper we studied the distribution of the zeros of $L_n$, that can be expressed in terms of nested radicals of $2$; it allow us to give an ordering for nested square roots of 2 expressed by (\ref{eq:prop2a}) thanks to a binary code employed in Informatics (the Gray code). In further papers in preparation we are trying to apply the results introduced here, in order to obtain formulas for the approximation of $\pi$.

Moreover, in future developments, it would be interesting to determine and study zeros of different classes of Lucas-Lehmer polynomials, for example modifying suitably the first term of the sequence.




\end{document}